\documentclass[12pt,reqno]{amsart}
\usepackage{a4,amssymb,amsthm,amscd,amsmath,verbatim,url,enumerate,mathdots, hyperref,
cleveref,mathtools}
\usepackage[pdftex,usenames,dvipsnames]{xcolor}
\usepackage{fancyhdr}

\title{A ruled residue theorem for function fields of elliptic curves}
\author{Karim Johannes Becher}
\author{Parul Gupta}
\author{Sumit Chandra Mishra}
\email{karimjohannes.becher@uantwerpen.be}
\email{parul.gupta@iiserpune.ac.in, parul.gupta@uantwerpen.be}
\email{sumitcmishra@gmail.com}
\address{University of Antwerp, Department of Mathematics, Middelheim\-laan~1, 2020 Antwerpen, Belgium.}
\address{Indian Institute of Science Education and Research (IISER) Pune, Dr.~Homi Bhabha Road, Pashan, Pune 411 008, India}
\address{Indian Institute of Science Education and Research (IISER) Mohali, Knowledge City, Sector 81, Mohali 140306, India}

\date{29.06.2023}

\newcommand{\qq}{\mathbb Q}

\newcommand{\nat}{\mathbb{N}} 
\newcommand{\zz}{\mathbb Z}
\newcommand{\rr}{\mathbb R}

\newcommand{\mc}[1]{\mathcal{#1}}
\newcommand{\mf}[1]{\mathfrak{#1}}

\newcommand{\mg}[1]{{#1}^{\times}}
\newcommand{\sq}[1]{{#1}^{\times 2}}

\newcommand{\ovl}{\overline}

\newcommand{\wt}[1]{\widetilde{{#1}}}
\newcommand{\ind}{\mathsf{ind}}

\newcommand{\car}{\mathsf{char}}

\newcommand{\s}{\sigma}
\newcommand{\Pol}{\mathsf{P}}

\renewcommand{\deg}{\mathsf{deg}}

\renewcommand{\max}{\mathsf{max}}
\renewcommand{\min}{\mathsf{min}}

\newcommand{\mfm}{\mf m}
\renewcommand{\setminus}{\smallsetminus}

\renewcommand{\leq}{\leqslant}
\renewcommand{\geq}{\geqslant}

\swapnumbers

\newtheorem*{thm*}{Theorem}

\newtheorem{thm}{Theorem}
\numberwithin{thm}{section}
\newtheorem{prop}[thm]{Proposition}
\newtheorem{cor}[thm]{Corollary}

\newtheorem{lem}[thm]{Lemma}

\theoremstyle{definition}

\newtheorem{ex}[thm]{Example}

\newtheorem{rem}[thm]{Remark}

\numberwithin{equation}{thm}
\renewenvironment{proof}{\par\noindent {\em Proof:}}{\hfill$\Box$\medskip}
\theoremstyle{plain}

\begin{document}
\maketitle

\begin{abstract}
It is shown that a valuation of residue characteristic different from $2$ and $3$ on a field $E$ has at most one extension to the function field of an elliptic curve over $E$, for which the residue field extension is transcendental but not ruled. 
The cases where such an extension is present are characterised.

\medskip
\noindent
{\sc{Classification (MSC 2010):}} 12F20, 12J10, 12J20, 14H05, 16H05

\medskip
\noindent
{\sc{Keywords:}} valuation, residue field extension, Gauss extension, rational function field, function field in one variable
\end{abstract}

%%%%%%%%%%%%%%%%%%%%%%%%%%%%%%%%%%%%%%%%%%%%%%%%%%%%%%%%%%%%%%
\section{Introduction}\label{introduction}%%%%%%%%%%%%%%%%%%%%
%%%%%%%%%%%%%%%%%%%%%%%%%%%%%%%%%%%%%%%%%%%%%%%%%%%%%%%%%%%%%%

In this article, we study extensions of a valuation on a field $E$ to function fields in one variable over $E$.
By a \emph{function field in one variable}, we mean a finitely generated field extension of transcendence degree one.

Let $F/E$ be a function field in one variable. 
We call $F/E$ \emph{rational} if $F=E(x)$ for some element $x\in F$, and we call $F/E$ \emph{ruled} if $F=E'(x)$ for some finite extension $E'/E$ and some element $x\in F$. 

We now consider a valuation $v$ on $E$ with arbitrary value group, denoted by~$vE$. 
By an \emph{extension of $v$ to $F$} we mean a valuation $w$ on $F$ with $w|_E=v$.
In 1983, J.~Ohm \cite{Ohm} proved that, if $F/E$ is ruled, then for any extension $w$ of $v$ to $F$, the  residue field extension $Fw/Ev$ is either algebraic or ruled. 
This is called the \emph{Ruled Residue Theorem}. 
In the case where $vE=\zz$, 
this result had been obtained already in 1967 by M.~Nagata \cite[Theorem 1]{Nag}.
What is the situation if one does not assume $F/E$ to be ruled?

An extension of $v$ to $F$ is called \emph{residually transcendental} if the residue field extension $Fw/Ev$ is transcendental.
Let $\Omega_v(F)$ denote the set of residually transcendental extensions of $v$ to $F$ and let $\Omega_v^\ast(F)$ denote the subset of those extensions $w$ for which $Fw/Ev$ is not ruled.
Expressed in these terms, the Ruled Residue Theorem says that $\Omega_v^\ast(F)=\emptyset$ when $F/E$ is ruled.

In the case where $vE=\zz$, 
it was observed in \cite[Corollary 3.9]{BGVG} that the set $\Omega_v^\ast(F)$ is  finite, and in \cite[Theorem 5.3]{BGr}, assuming further that $E$ is relatively algebraically closed in $F$, it was shown that 
$|\Omega_v^\ast(F)|\leq {g}+1$ where ${g}$ is the genus of $F/E$, and further that the inequality is strict when $F/E$ has prime divisor of degree $1$.
So for example, if $vE =\zz$ and $F/E$ is the function field of a conic or of an elliptic curve, then $|\Omega_v^\ast(F)|\leq 1$.

A first step toward extending these results to the case of a valuation $v$ with arbitrary value group was taken in \cite{BG21}, where it was shown that, if $F/E$ is the function field of a conic and $v(2)=0$, then $|\Omega_v^\ast(F)|\leq 1$.
The aim of this article is to establish the same conclusion in the case where $F/E$ is the function field of an elliptic curve and $v(6)=0$, without any further assumption on the value group of $v$.
This supports the expectation that the results from \cite{BGr} can be extended so as to eliminate the assumption on the value group $vE$.

We call the function field in one variable $F/E$ \emph{elliptic} if $F/E$ has genus $1$ and 
carries a prime divisor of degree $1$. 
In other terms, $F/E$ is elliptic if and only if it is the function field of an elliptic curve over $E$.
Since our results concern the case where $\car(E)$ is different from $2$ and $3$, we may present an elliptic function field in a particularly nice form.

Let $a,b \in E$. 
We associate the quantity $\Delta_{a,b}= 4a^3 +27b^2$, which is the discriminant of the cubic polynomial $X^3+aX+b$.  
If $\Delta_{a,b}\neq 0$, then $Y^2=X^3+aX+b$ is an elliptic curve over $E$, and its function field is given by $E(X)[\sqrt{X^3+aX+b}]$.
This is an elliptic curve in Weierstra{\ss} form, and if $\car{(E)}$ is different from $2$ and $3$, then every elliptic function field over $E$ can be presented in this way.

Let $v$ be a valuation on $E$ with $v(6) =0$. Hence $\car (Ev)$ and $\car (E)$ are different from $2$ and $3$. 
Let $F/E$ be an elliptic function field.
We will show that there can be at most one residually transcendental extension $w$ of $v$ to $F$ such that $Fw/Ev$ is not ruled. We will further characterize the situations when such an extension occurs in terms of conditions on $a,b$ and $\Delta_{a,b}$.  
We call $F/E$ \emph{of good reduction with respect to $v$} if there exist $a,b\in \mc{O}_v$ with $\Delta_{a,b}\in\mg{\mc O}_v$ such that $F\simeq E(X)[\sqrt{X^3+aX+b}]$, where $\mc{O}_v$ denotes the valuation ring of $v$.
Note that $F/E$ is a regular extension, so for every finite field extension $E'/E$, $F\otimes_E E'$ is a field, which we denote by $FE'$.
We call $F/E$ \emph{of potential good reduction with respect to $v$} if there exists a finite field extension $E'/E$ such that $FE'/E'$ is of good reduction with respect to some extension of $v$ from $E$ to $E'$. (We will see that $E'/E$ can then be chosen to be of degree at most $6$ and totally ramified with respect to $v$, whereby the extension of $v$ to $E'$ is unique.)

Our main result, \Cref{main_result}, can be stated as follows:
 \begin{thm*}
    Let $F/E$ be an elliptic function field and let $v$ be a valuation on $E$ with $v(6)=0$.  
    \begin{enumerate}[$(a)$]
        \item Assume that $F/E$ is of good reduction with respect to $v$. Then $|\Omega_v^\ast(F)| = 1$ and for $w \in \Omega_v^\ast(F)$ we have that $Fw/Ev$ is an elliptic function field. 
    
        \item Assume that $F/E$ is of potential good reduction but not of good reduction with respect to $v$.
Then $\Omega_v^\ast(F)=\emptyset $.
        \item  Assume that $F/E$ is not of potential good reduction with respect to $v$.
        Then $|\Omega_v^\ast(F)|\leq 1$.
       Moreover, if $|\Omega_v^\ast(F)| =  1$ then for $w \in \Omega_v^\ast(F)$ we have that $Fw/Ev$ is the function field of a smooth conic over $Ev$.
    \end{enumerate}
\end{thm*}

%%%%%%%%%%%%%%%%%%%%%%%%%%%%%%%%%%%%%%%%%%%%%%%%%%%%%%%%%%%%%%%%%%%%%%%%
\section{Valuations on simple extensions}\label{valonRFF}%%%%%%%%
%%%%%%%%%%%%%%%%%%%%%%%%%%%%%%%%%%%%%%%%%%%%%%%%%%%%%%%%%%%%%%%%%%%%%%%%

In this section we revisit some standard facts from valuation theory, in particular concerning extensions of valuations.
This includes the Fundamental Inequality. 
Our main reference is \cite[Chap.~2-3]{EP}.

Let $K$ be a field.
Consider a valuation $v$ on $K$.
We denote by $\mc{O}_v$ the valuation ring, by $\mfm_v$ its maximal ideal, by $Kv$ the residue field $\mc{O}_v/\mfm_v$, and by $vK$ the value group of $v$.
We denote the group operation in $vK$ additively. 
Note that $vK$ can be identified with the (multiplicatively written) quotient group $\mg{K}/\mg{\mc O}_v$ by taking the  order relation $\leq$ given by letting $a\mg{\mc O}_v\leq b\mg{\mc O}_v$ for $a,b\in\mg{K}$ whenever $a\mc{O}_v\subseteq b{\mc O}_v$. 
For $x\in {\mc O_v}$ we denote the residue class $x+\mfm_v$ in $Kv$
 by $xv$, or just by 
 $\ovl{x}$ if the context makes this unambiguous.
A valuation $v'$ on $K$ is \emph{equivalent to $v$} if $\mc{O}_v=\mc{O}_{v'}$, or in other terms, if $v'=\tau\circ v$ for an order preserving isomorphism $\tau:vK\to v'K$.

We will now recall some facts on extensions of valuations along a field extension.
Let $v$ be a valuation on $K$. Let $L/K$ be a field extension.
An \emph{extension of $v$ to $L$} is a valuation $w$ on $L$ such that $vK$ is an ordered subgroup of $wL$ and $w|_K=v$. By Chevalley's Theorem \cite[Section 3.1]{EP}, $v$ always extends to a valuation on $L$.
In this article, we will focus on extensions $w$ of $v$ to $L$ for which the quotient group $wL/vK$ (naturally isomorphic to $\mg{L}/\mg{K}\mg{\mc{O}}_w$) is finite.
Note that any ordered abelian group $\Gamma$ is torsion-free and therefore embeds uniquely into its divisible closure
$\Gamma\otimes_\zz\qq$, and that further the group ordering of $\Gamma$ extends uniquely to $\Gamma\otimes_\zz\qq$, whereby we view $\Gamma\otimes_\zz\qq$ as an ordered group without any ambiguity.
Hence, an extension $w$ of $v$ to $L$ where $wL/vK$ is torsion (e.g.~finite) can be viewed in a unique way as a map to $vK \otimes_\zz \qq$.
Consider now two extensions $w$ and $w'$ of $v$ to $L$ such that $wL/vK$ and $w'L/vK$ are torsion.
If $w$ and $w'$ are equivalent as valuations, then viewing them (in the only possible way) as maps to $vK\otimes_\zz\qq$, they coincide, and hence we consider them as \emph{equal}. We therefore call $w$ and $w'$ \emph{distinct} if they are not equivalent.

This applies in particular to all extensions of $v$ to $L$ when $[L:K]<\infty$.
In fact, if $[L:K]<\infty$, then for any extension $w$ of $v$ to $L$, we have $[wL:vK]\leq [L:K]< \infty$, by \cite[Cor.~3.2.3]{EP}, and hence $wL=[wL:vK]^{-1}\cdot vK$ as a subgroup of $vK\otimes_\zz \qq$.
We can now formulate the Fundamental Inequality, which is crucial in understanding the mutual relations between distinct extensions  of a valuation in  a finite field extension.

\begin{thm}[Fundamental Inequality]
\label{funda_ineq}
Assume that $L/K$ is a finite field extension. 
Let $r\in\nat$ and let $w_1,\dots, w_r$ be distinct extensions of  $v$ to $L$.
Then
$$ \displaystyle \sum^{r}_{i=1} [w_iL:vK] \cdot [Lw_i:Kv] \leq [L:K].$$
\end{thm}

\begin{proof}
See \cite[Theorem 3.3.4]{EP}.
\end{proof}

\begin{cor}
\label{funda_quad}
    Assume that $L/K$ is a quadratic field extension.
    Let $w$ be an extension of $v$ to $L$. Then $[wL:Kv]\cdot [Lw:Kv]\leq 2$, and if this is an equality, then $w$ is the unique extension of $v$ to $L$.
\end{cor}
\begin{proof}
    Since $[L:K]=2$, this follows by \Cref{funda_ineq}.
\end{proof}

Let $w$ be an extension of $v$ to $L$.
We call $w$ \emph{unramified in $L/K$} if $Lw/Kv$ is separable and $wL=vK$, and otherwise we call $w$ \emph{ramified in $L/K$}.

For $f=a_nT^n+\cdots + a_0 \in \mc O_v[T]$, we set  $\ovl{f}= \ovl{a_n} T^n + \cdots +\overline{a_1}T + \overline{a_0} \in Kv[T]$.

\begin{lem}\label{simple_ext}
    Let $L/K$ be a finite separable field extension. 
    Let $\alpha\in L$ be such that $L=K[\alpha]$ and let $f\in\mc{O}_v[T]$ irreducible in $K[T]$ and with $f(\alpha)=0$.
    Assume that $\ovl{f}\in Kv[T]\setminus Kv$ and let $q$ be an irreducible factor of $\ovl f$ in $Kv[T]$.
    Then there exists an extension $w$ of $v$ to $L$ such that $\alpha\in\mc{O}_w$ and $q(\ovl{\alpha})=0$ in $Lw$. 
    Furthermore, given such an extension $w$ of $v$ to $L$, if $\ovl{\alpha}\in Lw$ is a simple root of $\ovl f$, then the extension $w$ of $v$ is unramified and $Lw=Kv[\ovl{\alpha}]$.
\end{lem}

\begin{proof}
Let $N$ be the splitting field of $f$ over $L$.
Let $G$ be the Galois group of $N/K$.
Let $\wt w$ be an extension of $v$ to $N$.
We write 
$$f=c\prod_{i=1}^n (T-\alpha_i)$$
with $c\in \mg{K}$, $n\in\nat$ and $\alpha_1, \ldots, \alpha_n\in N$ such that 
$\wt{w}(\alpha_1)\geq \dots\geq \wt{w}(\alpha_n)$. 
Let $r\in\{0,\dots,n\}$ be such that $\alpha_1,\dots,\alpha_r\in\mc O_{\wt w}$ and  $\alpha_{r+1}^{-1}, \ldots, \alpha_n^{-1} \in \mf m_{\wt w} $.
Note that
$$f = c \alpha_{r+1}\cdots\alpha_n  \prod_{i=1}^r(T-\alpha_i)\prod_{i={r+1}}^n(\alpha_i^{-1}T -1) $$
and  
$$\ovl{f} = (-1)^{n-r} (\ovl{c \alpha_{r+1}\cdots\alpha_n }) \prod_{i=1}^r(T-\ovl{\alpha_i}). $$
Thus $\deg(\ovl{f})=r$ and $\ovl{\alpha_1}, \ldots, \ovl{\alpha_r} \in N{\wt w}$ are roots of $\ovl{f}$.
Since $q$ is an irreducible factor of $\ovl{f}$ in $Kv[T]$, there exists $\s\in G$ such that $\s(\alpha)\in\mc{O}_{\wt{w}}$ and $q(\ovl{\s(\alpha)})=0$ in $N\wt{w}$.
Set $w = {\wt w} \circ \sigma$.
Then $\alpha\in\mc{O}_w$ and  $q(\ovl{\alpha})= 0$ in $L{w}$.

Assume now that $\ovl{\alpha}$ in $Lw$ is a simple root of $\ovl{f}$.
 In particular, $\ovl \alpha$ is separable over $Kv$.
Let $(K^\ast, v^\ast)$ and $(L^\ast, w^\ast)$ denote the henselizations of $(K,v)$ and $(L,w)$, respectively (see \cite[page 121]{EP}). 
We may view $K^\ast$ as a subfield of $L^\ast$ in such way that $w^\ast|_{K^\ast}=v^\ast$.
By \cite[Theorem 5.2.5]{EP}, we have  $K^\ast{v^\ast} =Kv$, ${v^\ast}K^\ast = vK$,  $L^\ast{w^\ast} =Lw$ and ${w^\ast}L^\ast =wL$.
In particular, by \Cref{funda_ineq}, we have 
$$[wL:vK]\cdot  [Lw:Kv]\leq [L^\ast:K^\ast]\,.$$
By \cite[Theorem 5.2.2]{EP}, we have $L^\ast = K^\ast[\alpha]$ and $w^\ast$ is the unique extension of $v^\ast$ to $K^\ast[\alpha]$. 
 Let $g$ be the minimal polynomial of $\alpha$ over $K^\ast$. Then $\deg(g)=[L^\ast: K^\ast]$.
Since $q\in Kv[T]$ is irreducible and $q(\ovl \alpha)=0$, we obtain that $q$ is separable.
Since $K^\ast$ is henselian, we obtain by \cite[Theorem 4.1.3 (2)]{EP}  that $g\in \mc{O}_{v^\ast}[T]$ and $\ovl{g} = q$.
Hence  $$[Kv[\ovl \alpha]:Kv]=\deg(q)=\deg(\ovl g)=\deg(g)=[L^\ast:K^\ast]\geq [Lw:Kv]\,.$$
As $\ovl \alpha\in Lw$, we conclude that $Lw=Kv[\ovl\alpha]$.
In particular, $Lw/Kv$ is separable and $[Lw:Kv]=[L^\ast:K^\ast]$, whereby $[wL:vK]=1$.
Therefore $w$ is an unramified extension of $v$.
\end{proof}

Let $p$ be a prime number. For $a\in \mg K\setminus K^{\times p}$, the polynomial $X^p-a$ is irreducible in $K[T]$, and we denote  by $K[\sqrt[p]{a}]$ its root field over $K$.
For $a\in K^{\times p}$, we set $K[\sqrt[p]{a}]=K$.

\begin{cor}
\label{2.2 Generalization_without_roots}
 Let $p$ be a prime number such that $v(p)=0$.
 Let $a \in K^{\times}\setminus K^{\times p}$ be such that $v(a) \in p vK$. 
 Then 
 $aK^{\times p} \cap \mathcal{O}^{\times}_v \neq \emptyset$ and 
 any extension of $v$ to $K[\sqrt[p]{a}]$ is unramified.
 Furthermore, fixing $u \in aK^{\times p} \cap \mathcal{O}^{\times}_v$, one has:
 \begin{enumerate}[$(i)$]
\item  If $\overline{u}\notin Kv^{\times p}$, then $v$ extends uniquely to $K[\sqrt[p]{a}]$, and the residue field of this extension is $Kv[\sqrt[p]{\overline{u}}]$.

\item  If $\overline{u}\in Kv^{\times p}$, then for any extension $w$ of $v$ to $K[\sqrt[p]{a}]$, the residue field extension $K[\sqrt[p]{a}]{w}/Kv$ is either trivial or given by adjoining a primitive $p$th root of unity.
\end{enumerate}
\end{cor}

\begin{proof}
Since $v(a) \in pvE$, we can find an element $x\in K^{\times}$ with $v(a)=pv(x)$.
We set $u=ax^{-p}$. Then $u\in aK^{\times p} \cap \mathcal{O}^{\times}_{v}$. In particular, this set is not empty.
Furthermore $K[\sqrt[p]{a}]=K[\sqrt[p]{u}]$.

Let $w$ be an extension of $v$ to $K[\sqrt[p]{a}]$.
Since $\car{(Kv)} \neq p$, the polynomial $T^p - \ovl{u}\in Kv[T]$ is separable, so we obtain by \Cref{simple_ext} that $wK[\sqrt[p]{a}]=vK$ and that $w$ is an unramified extension of $v$.

$(i)$\, Assume that $\ovl{u} \notin Kv^{\times p}$.
Then $K[\sqrt[p]{a}]w = Kv[\sqrt[p]{\ovl{u}}]$, by \Cref{simple_ext}, and $[K[\sqrt[p]{a}]w: Kv] =p=[K[\sqrt[p]{a}]:K]$.
Hence it follows by \Cref{funda_ineq}  that $w$ is the unique extension of $v$ to $K[\sqrt[p]{a}]$.

$(ii)$\, Assume that $\ovl{u} \in Kv^{\times p}$. Let $\vartheta\in Kv$ be such that $\ovl{u}=\vartheta^p$.
Then it follows by \Cref{simple_ext} that $K[\sqrt[p]{a}]{w} = Kv[\beta]$ for a root $\beta$ of $T^p-\ovl{u}\in Kv[T]$.
Since $\beta^p = \vartheta^p$, it follows that $\beta= \rho \vartheta$ for some root $\rho$ of $T^p-1$.
Since $\vartheta \in Kv$, we obtain that $K[\sqrt[p]{a}]{w} = Kv[\rho]$.
\end{proof}

The following special case of \Cref{2.2 Generalization_without_roots} for a quadratic extension will be used multiple times in this article.

\begin{cor}
\label{BG2.2}
Assume that $v(2)=0$.
 Let $a \in K^{\times}\setminus K^{\times 2}$ with $v(a) \in 2 vK$ and 
let $w$ be an extension of $v$ to $K[\sqrt{a}]$.
Then the extension $w$ of $v$ is unramified and $aK^{\times 2} \cap \mathcal{O}^{\times}_v \neq \emptyset$.
Furthermore, given $u \in aK^{\times 2} \cap \mathcal{O}^{\times}_v$, if $\ovl{u} \notin \sq{Kv}$, then $K[\sqrt{a}]w = Kv[\sqrt {\ovl{u}}]$ and $w$ is the unique extension of $v$ to $K[\sqrt{a}]$, and otherwise $K[\sqrt{a}]w = Kv$.
\end{cor}

\begin{proof}
It follows from \Cref{2.2 Generalization_without_roots} for $p=2$.
\end{proof}

\begin{lem}
\label{f_and_g}
Let $p$ be a prime number such that $v(p)=0$.
Let $a, b \in K$ be such that $v(a)<v(b)$.
Let $L=K[\sqrt[p]{a+b}]$ and let $w$ be  an extension  of $v$ to $L$.
Then one of the following three conditions holds:
\begin{enumerate}[$(i)$]
    \item $Lw=K[\sqrt[p]{a}]w'$ for an extension $w'$ of $v$ to $K[\sqrt[p]{a}]$.
    \item $Lw=Kv$.
    \item $Lw=Kv[\rho]$ for a primitive $p$th root of unity $\rho$.
\end{enumerate}
\end{lem}

\begin{proof}
Since $v(a)<v(b)$, we have $v(a+b)=v(a)$. 
If $v(a)\notin p vK$ then for any extension $w'$ of $v$  to $K[\sqrt[p]{a}]$, we have  $[wL:vK] =p= [{w'}K[\sqrt[p]{a}]:vK]$ and hence $Lw=Kv=K[\sqrt[p]{a}]{w'}$, by \Cref{funda_ineq}.

Now suppose that $v(a)\in p vK$.
Let $x\in K$ be such that $ax^{p}\in \mathcal{O}^{\times}_{v}$.
Set $u =ax^{p}$.
Then $\overline{(a+b)x^{p}}=\overline{ax^{p}}=\overline{u}$ in $Kv$.
If $\overline{u}\notin Kv^{\times p}$ then by \Cref{2.2 Generalization_without_roots} $(i)$, for the unique extension $w'$ of $v$ to $K[\sqrt[p]{a}]$ we have
$Lw=Kv[\sqrt[p]{\overline{u}}] = K[\sqrt[p]{a}]{w'}$.

If $\overline{u}\in Kv^{\times p}$, then by \Cref{2.2 Generalization_without_roots} $(ii)$, we have either $Lw=Kv$ or $Lw=Kv[\rho]$ for a primitive $p$th root of unity $\rho$.
\end{proof}

%%%%%%%%%%%%%%%%%%%%%%%%%%%%%%%%%%%%%%%%%%%%%%%%%%
\section{Residually transcendental extensions} %%%
%%%%%%%%%%%%%%%%%%%%%%%%%%%%%%%%%%%%%%%%%%%%%%%%%%

Let $E$ be a field and  $v$ a valuation on $E$.
In this section, we consider residually transcendental extensions of $v$ to certain function fields in one variable over $E$.

Let $F/E$ be a function field in one variable over $E$.
An extension $w$ of $v$ to $F$ is called \emph{residually transcendental} if the residue field extension $Fw/Ev$ is transcendental. 
The following well-known statement gives a  standard example of a residually transcendental extension of $v$ from $E$ to the rational function field $E(X)$, given in terms of the variable $X$.

\begin{prop}\label{gaussextdef}
There exists a unique valuation $w$ on $E(X)$ with $w|_E=v$, $w(X) = 0$ and such that the residue $\ovl{X}$  of $X$ in $E(X)w$ is transcendental over $Ev$. 
For this valuation $w$, we have that $E(X)w =Ev(\ovl{X})$ 
and $wE(X) = vE$.
For $n\in\nat$ and $a_0,\dots,a_n\in E$, we have 
$w(\mbox{$\sum_{i=0}^n$}a_iX^i)= \min\{v(a_0),\dots,v(a_n)\}$.  
\end{prop}

\begin{proof} 
See \cite[Corollary 2.2.2]{EP}.
\end{proof}

The valuation on $E(X)$ defined in \Cref{gaussextdef} is called the \emph{Gauss extension of $v$ to $E(X)$ with respect to $X$}\index{Gauss extension}. 
More generally, we call a valuation on $E(X)$ a \emph{Gauss extension of $v$ to $E(X)$} if it is the Gauss extension with respect to $Z$ for some  $Z\in E(X)$ with $E(Z)=E(X)$.

\begin{cor}
Let $F/E$ be a function field in one variable.
Let $w$ be a residually transcendental extension of $v$ to $F$.
Then $Fw/Ev$ is a function field in one variable.
\end{cor}
\begin{proof}
Since $w$ is residually transcendental extension of $v$ to $F$,
there exists $Z\in \mg{\mc O}_w$ such that $\ovl Z$ is transcendental over $Ev$.
By \Cref{gaussextdef}, $w|_{E(Z)}$ is the Gauss extension of $v$ with respect to $Z$.
Since $F/E(Z)$ is a finite field extension, using \Cref{funda_ineq} we see that $Fw$ is a finite extension of $Ev(\ovl Z)$.
Thus $Fw/Ev$ is a function field in one variable.
\end{proof}

For a residually transcendental extension $w$ of $v$ to $F$, we define 
$$\ind(w,F/E)= \min\{[F: E(Z)]\,\mid\, Z \in \mg{\mc O}_w\mbox{ and~} \ovl{Z} \mbox{ is transcendental over } Ev \},$$
and call this the \emph{Ohm index of $w$ for $F/E$}.

\begin{rem}\label{Ohmindexrem}
Note that $\ind(w,F/E)=1$ if and only if $F$ is a rational function field over $E$ and $w$ is a Gauss extension of $v$ to $F$.
\end{rem}

An element $Z\in \mg{\mc{O}}_w$ is called an \emph{Ohm element of $w$ for $F/E$} if
$\ovl{Z}$ is transcendental over $Ev$
and $[F:E(Z)] =\ind(w,F/E)$.
It is clear that for any residually transcendental extension of a valuation to a function field in one variable there exists an Ohm element; see \cite[Lemma]{BG21}.
We recall Ohm's Ruled Residue Theorem, which is the starting point of this investigation as well as a crucial ingredient in obtaining related results.

\begin{thm}[Ohm]\label{RRT}
Let $w$ be a residually transcendental extension of $v$ to $E(X)$.
Let $\ell$ be the relative algebraic closure of $Ev$ in $E(X)w$. 
Then $E(X)w$ is a rational function field over $\ell$.
More precisely,  $E(X)w = \ell( \ovl{Z})$ for any Ohm element $Z$ of $w$ over $E$.
\end{thm}

\begin{proof}
See \cite[Theorem 3.3]{Ohm} and its proof.
\end{proof} 

\begin{cor}
\label{RRT_corollary}
    Let $F/E$ be a function field in one variable and $v$  valuation on $E$.
    If $F/E$ is ruled, then $Fw/Ev$ is ruled for every residually transcendental extension $w$ of $v$ to $F$.
\end{cor}

\begin{proof}
Assume that $F/E$ is ruled. 
 Hence $F=E'(X)$ for a finite field extension $E'/E$ and some $X\in F$ transcendental over $E'$.
    Let $w$ be a residually transcendental extension of $v$ to $F$.
    Then $Fw/E'w$ is ruled by \Cref{RRT}, and as $[E'w:Ev]\leq [E':E]<\infty$, we conclude that $Fw/Ev$ is ruled.
\end{proof}

We will now draw some consequences from the preceding statements, in particular from \Cref{RRT}.
This will provide us with various sufficient conditions for ruledness of certain residue field extensions of valuations on a function field.
These statements will later be applied to the case of an elliptic function field.

\begin{lem} \label{equal_values_ruled}
Let $p$ be a prime number with $v(p)=0$.
Let $f,g\in E(X)$ and
let $F=E(X)[\sqrt[p]{f+g}]$ and $F'=E(X)[\sqrt[p]{f}]$. 
Suppose that $F'w'/Ev$ is ruled for every residually transcendental extension $w'$ of $v$ to $F'$.
Let $w$ be a residually transcendental extension of $v$ to $F$ such that $w(f)<w(g)$.
Then $Fw/Ev$ is ruled.
\end{lem}
\begin{proof}
By \Cref{RRT}, $E(X)w/Ev$ is ruled.
Hence also $E(X)w[\rho]/Ev$ is ruled for a primitive $p$th root of unity $\rho$.
In view of the hypotheses on $w$ and $F'$, the statement now follows by \Cref{f_and_g}.
\end{proof}

\begin{prop}
\label{equal_values}
Let $n \in \nat\setminus\{0,1\}$ and let $p$ be a prime number with $v(p)=0$.

Let $a,b\in E$ and $F=E(X)[\sqrt[p]{X^n+aX+b}]$. 
Let $w$ be a residually transcendental extension of $v$ to $F$.
If $w(X^n)\neq w(aX+b)$ then $Fw/Ev$ is ruled.
\end{prop}
\begin{proof}
Since $E(X)[\sqrt[p]{aX+b}]/E$ and $E(X)[\sqrt[p]{X^n}]/E$ are rational function fields in one variable, the statement follows by \Cref{RRT} and \Cref{equal_values_ruled}.
\end{proof}

 \begin{lem}
 \label{scalarext-noresidueext_for_p}
Assume that $v(p)=0$. 
Let $F/E$ be a function field in one variable  and $w$ be an extension of $v$ to $F$ such that $Fw/Ev$ is non-ruled.
Let  $b\in \mg E\setminus E^{\times p}$  and $F' =F[\sqrt[p]{b}]$. 
Assume that there exists $u \in bF^{\times p}\cap \mathcal{O}^{\times}_w$ such that $\ovl{u}\in Fw^{\times p}$.
Then there exists an extension $w'$ of $w$ to $F'$ such that $F'w'/Ev$ is non-ruled.
\end{lem}

\begin{proof}
Let $u \in  bF^{\times p}\cap \mg{\mc O}_w$ such that $\ovl {u} \in Fw^{\times p} $.
Then $T^p -u \in \mc O_w[T]$ is a monic irreducible polynomial and
the residue polynomial $T^p -\ovl{u} \in Fw[T]$ has a simple root in $Fw$. Thus by \Cref{simple_ext}, there is an extension $w'$ of $w$  to $F'$ such that $F'w' =Fw$ and $wF=w'F'$. 
\end{proof}

%%%%%%%%%%%%%%%%%%%%%%%%%%%%%%%%%%%%%%%
\section{Function fields of conics} %%%
%%%%%%%%%%%%%%%%%%%%%%%%%%%%%%%%%%%%%%%

Let $E$ be a field and $v$ a valuation on $E$.
Given a function field in one variable $F/E$, we want to count and describe the extensions $w$ of $v$ to $F$ for which the residue field extension $Fw/Ev$ is a non-ruled function field in one variable.
In the lack of a method to approach this problem in general, one may consider this problem for special classes of function fields $F/E$. 
In this section, we will look at quadratic extensions of the rational function field $E(X)$ with residually transcendental extension of $v$ having special properties. We will further recall some results from \cite{BG21} for function fields of conics. These results will turn out to be useful later for the case of an elliptic function field. 

\begin{lem}\label{every-ele-trans}
Assume that $v(2) =0$.
Let $f\in E[X]$ and $F=E(X)[\sqrt{f}]$. 
Let $w$ be a residually transcendental extension of $v$ to $F$ and let $\ell$ be the relative algebraic closure of $Ev$ in $Fw$. Assume that $Fw \neq E(X)w$. Then $\ell \subseteq E(X)w$ if and only if, for every $\phi \in E(X)$ with $w(f\phi^2) =0$, we have that $\ovl{f\phi^2}$  is transcendental over~$Ev$.
\end{lem}

\begin{proof}
Since $Fw \neq E(X)w$, by \Cref{funda_quad}, we have that 
$$[Fw:\ell E(X)w]\cdot [\ell E(X)w : E(X)w]=[Fw:E(X)w]= 2\,.$$
It follows by \Cref{BG2.2} that there exists $\phi\in E(X)$ with $w(f\phi^2)=0$ and 
$$F{w}=E(X){w}\Big[\sqrt{\overline{f\phi^2}}\Big]\,.$$ 
If $\ovl{f\phi^2}$ is algebraic over $Ev$, then $\sqrt{\ovl{f\phi^2}}\in \ell$, whence
$\ell E(X)w=Fw\neq E(X)w$, that is $\ell\not\subseteq E(X)w$.

Assume now conversely that $\ell \not \subseteq E(X)w$.
Then $[\ell E(X)w:E(X)w]=2$ and $Fw = \ell E(X)w$.
Let $\ell_0 = \ell \cap E(X)w$.
 Then $[\ell : \ell_0] =2$. Hence there exists $\gamma \in \ell \setminus \ell_0$ with $\gamma^2\in \ell_0$. 
Hence we may choose $\phi \in E(X)$ with $w(f\phi^2)=0$ in such a way that $\gamma^2 = \ovl{f \phi^2 }$ and with this choice, $\ovl{f\phi^2}$ is algebraic over $Ev$.
\end{proof}

For a separable quadratic field extension $E'/E$ and a residually transcendental extension $w$ of $v$ to $E'(X)$,
the relation between the Ohm indices of $w$ and $w|_{E(X)}$ was 
described in \cite[Lemma 2.8]{BG21}.
We now obtain an analogue for the case of a quadratic extension $F/E(X)$ such that $F/E$  is rational.

\begin{prop}
 \label{a_version_of_RRT_2.8}
  Assume that $v(2) =0$.
   Let $F= E(\sqrt{X})$. 
   Let $w$ be a residually transcendental extension of $v$ to $F$.
   Let $\ell$ be the relative algebraic closure of $Ev$ in $Fw$.
 Assume that $Fw \neq E(X)w$ and $\ell \subseteq E(X)w$.
Then there exists an element $\phi \in E(X)^{\times}$ such that $\sqrt{X}\phi$ is an Ohm element of $w$ for $F/E$ and  $E(X)w = \ell\big(\ovl{{X}\phi^2}\big)$. 
\end{prop}

\begin{proof} 
 Let $\theta \in F$ be an Ohm element of $w$ for $F/E$.
 Write  $\theta = \frac{\theta_1}{\theta_2}$ with coprime polynomials $\theta_1, \theta_2\in E[\sqrt{X}]$ such that $\theta_1$ is monic.
Let $f_0, f_1, g_0, g_1\in E[X]$ be such that $\theta_1=f_0+\sqrt{X}f_1 $ and $\theta_2= g_0+\sqrt{X}g_1$.
Since $\theta$ is an Ohm element of $w$ for $F/E$, we have $\ind(w,F/E)=[F : E(\theta)]$, and
by \cite[Proposition 2.1]{BG21}, we have 
\begin{align*}
[F:E(\theta)] &= \max \{\deg_{\sqrt{X}}(\theta_1), \deg_{\sqrt{X}}(\theta_2)\}\\ 
&=  \max \{\deg_{\sqrt{X}}(f_0), \deg_{\sqrt{X}}(\sqrt{X}f_1), \deg_{\sqrt{X}}(g_0), \deg_{\sqrt{X}}(\sqrt{X}g_1) \} .
\end{align*}
Assume that $w(f_0)=w(\sqrt{X}f_1)$.
By \Cref{every-ele-trans}, we get that $\overline{\sqrt{X}f_1f^{-1}_0}$ is transcendental over $Ev$.
By \cite[Proposition 2.1]{BG21},  
 $$[F:E(\sqrt{X}f_1f^{-1}_0)] \leq \max\{\deg_{\sqrt{X}}(f_0), \deg_{\sqrt{X}}(\sqrt{X}f_1)\} \leq \ind(w,F/E).$$
We conclude that $\sqrt{X}f_1f^{-1}_{0}$ is an Ohm element of $w$ for $F/E$.
Similarly, if $w(g_0)=w(\sqrt{X}g_1)$, then
$\sqrt{X}g_1g^{-1}_{0}$ is an Ohm element of $w$ for $F/E$.

We now assume that $w(f_0)\neq w(\sqrt{X}f_1)$ and
$w(g_0)\neq w(\sqrt{X}g_1)$.
In this case, the property of $\theta$ to be an Ohm element of $w$ for $F/E$
is not affected if we replace $\theta_1$ and $\theta_2$ by the element of smaller $w$-value among the pairs $(f_0,\sqrt{X}f_1)$ and $(g_0,\sqrt{X}g_1)$, respectively.
We may further r eplace $\theta$ by $\theta^{-1}$, if necessary.
After these changes, if required, $\theta$ is of the form $\phi$ or
$\sqrt{X}\phi$ for some $\phi \in E(X)^{\times}$.

Since $\ell(\ovl{\phi}) \subseteq E(X){w}$ and $Fw\neq E(X)w$, we have $Fw\neq \ell(\ovl{\phi})$.
As $F/E$ is a rational function field in one variable, we conclude by \Cref{RRT} that $\phi$ cannot be an Ohm element of $w$ for $F/E$.
Therefore we have $\theta =\sqrt{X}\phi$ and $F{w} =\ell(\ovl{\theta})$.
Since $\ell(\ovl{\theta}^2) \subseteq E(X){w}$ and $[F{w}: E(X){w}] =2 = [\ell(\ovl{\theta}): \ell(\ovl{\theta}^2)]$, we get that $E(X){w} = \ell(\ovl{\theta}^2) =\ell (\ovl{X\phi^2})$. 
\end{proof}

We apply the above result to quadratic extensions of $E(X)$ and obtain sufficient conditions for $Fw/Ev$ being ruled.

\begin{cor}
\label{ruled}
Assume that $v(2) =0$.
Let $f\in E[X]$ and $F=E(X)[\sqrt{f}]$. 
Let $w$ be a residually transcendental extension of $v$ to $F$ and let $\ell$ be the relative algebraic closure of $Ev$ in $Fw$. Assume that $wE(X)=wE(f)$ and $E(X)w\subseteq\ell E(f)w$.
Then $Fw/Ev$ is ruled. 
\end{cor}
\begin{proof}
If $Fw=E(X)w$ then $Fw/Ev$ is ruled by \Cref{RRT}. Thus we may assume that $Fw \neq E(X)w$. We note that in this case, $[Fw:E(X)w]=2$ and $wF =wE(X)$ by \Cref{funda_quad}.
If $\ell\not\subseteq E(X)w$, then it follows by \cite[Lemma 3.2]{BG21} that $Fw/Ev$ is ruled.
Hence we may assume now that $\ell\subseteq E(X)w$. 

As $wF=wE(X) = wE(f)$ and $f\in \sq{F}$, we have $w(f) \in2wE(X) = 2wE(f)$.
By \Cref{every-ele-trans}, for every $\phi\in E(X)$ such that $w(f\phi^2) =0$, we get that $\ovl{f\phi^2}$ is transcendental over $Ev$.
Let $\ell'$ be the relative algebraic closure $Ev$ in $E(\sqrt{f})w$.
Since $E(f) \subseteq E(X)$, we have that $\ovl{f\phi^2}$ is transcendental over $Ev$ for every $\phi\in E(f)$ with $w(f\phi^2) =0$. 
Hence  applying \Cref{every-ele-trans} to the field extension $E(\sqrt f) /E(f)$, we obtain that $\ell' \subseteq E(f)w$.
By \Cref{a_version_of_RRT_2.8}, we obtain that $E(f)w = \ell'(\ovl{f\phi^2})$ for some $\phi \in E(f)^{\times}$ with $w(f\phi^2)=0$.

In view of the hypothesis and using that $\ell'\subseteq\ell \subseteq E(X)w$, we conclude that $$E(X)w = \ell E(f)w = \ell\Big(\ovl{f\phi^2}\Big)\,.$$
By \Cref{BG2.2} we get that $$Fw =E(X)w\Big[ \sqrt{\ovl{f\phi^2}}\Big] = \ell \Big(\sqrt{\ovl{f\phi^2}}\Big)\,,$$ whereby $Fw/Ev$ is ruled.
\end{proof}

\begin{cor}
\label{ruled_applicable}
Assume that $v(2) =0$.
Let $w$ be a residually transcendental extension of $v$ to $E(X)$.
Let $f\in E[X]$. 
Let $\theta \in E(X)$ be such that $E(X) = E(f)[\theta]$ and let $\Pol_\theta \in E(f)[T]$ be the minimal polynomial of $\theta$ over $E(f)$.
Assume that $ \Pol_\theta  \in \mc{O}_w[T]$ and that
$\ovl{\theta}$ is a simple root of $\ovl{\Pol_\theta}$.
Then $wE(X) =wE(f)$ and $E(X)w =E(f)w[\ovl{\theta}]$.
Furthermore, if $\ovl{\theta}$ is  algebraic over $Ev$, then $E(X)[\sqrt{f}]w'/Ev$ is ruled for any extension $w'$ of $w$ to $E(X)[\sqrt{f}]$.
\end{cor}
\begin{proof}
The first statement follows by \Cref{simple_ext}.
The second statement follows by \Cref{ruled}.
\end{proof}

\begin{prop}
 \label{u_algebraic_implies_ruled}
 Assume that $v(2)=0$.
 Let $u\in E(X)$ and $F=E(X)[\sqrt{uX}]$. 
Let $w$ be a residually transcendental extension of $v$ to $F$ with  $w(u)=0$
and such that $\ovl{u}$ is algebraic over $Ev$. 
Then $Fw/Ev$ is ruled. 
 \end{prop}

\begin{proof}
If $Fw=E(X)w$ then $Fw/Ev$ is ruled by \Cref{RRT}. Thus we may assume that $Fw \neq E(X)w$. Then $[Fw:E(X)w]=2$ and $wF =wE(X)$ by \Cref{funda_quad}. 
Hence $w(X)=w(uX)\in 2wF =2wE(X)$.
Let $\ell$ denote the relative algebraic closure of $Ev$ in $Fw$.
In view of \cite[Lemma 3.2]{BG21}, we may further assume that $\ell\subseteq E(X)w$.

Let $w'$ be an extension of $w|_{E(X)}$ to $E(\sqrt{X})$.
For every $\phi\in E(X)^{\times}$ with $w(X\phi^2)=0$, it follows by \Cref{every-ele-trans} that $\ovl {u X\phi^2}$ is transcendental over $Ev$, whereby $\ovl {X\phi^2} $ is transcendental over $Ev$, because $\ovl{u}\in \ell$. 
Using \Cref{every-ele-trans} now for the extension $E(\sqrt{X})/E(X)$, we get that $\ell$ is relatively algebraically closed in $E(\sqrt{X})w'$.

By \Cref{a_version_of_RRT_2.8},  there exists $\phi \in E(X)$ with $w(X\phi^2) =0$ such that, for $\vartheta=\ovl{X\phi^2}$, we have   $E(X)w=\ell(\vartheta)$.
Hence using \Cref{BG2.2} and because $\ovl{u}\in\ell$, we get that  
$$Fw = E(X)w[\sqrt{\ovl{u}\vartheta}]=\ell(\vartheta)[\sqrt{\ovl{u}\vartheta}]=\ell (\sqrt{\ovl{u}\vartheta})\,.$$  Hence $Fw/Ev$ is ruled.
\end{proof}

We conclude this section by a short discussion of our problem in the case of function fields of conics.

Recall that we assume that $\car(E) \neq 2$.
Consider $a,b\in \mg E$.
We associate the plane affine conic 
$$\mathcal C_{a,b}:Y^2 = aX^2 +b\,\,.$$
We observe that this curve is smooth and that its function field over $E$ is given by $E(X)[\sqrt{aX^2 +b}]$. 
We recall the following basic fact:

\begin{prop} \label{conicsplit}
Let  $a,b \in \mg E$. The following are equivalent:
\begin{enumerate}[$(i)$]
    \item $ \mc C_{a,b}$ has a rational point over $E$.
    \item $E(X) [\sqrt{aX^2 +b}]/E$ is a rational function field.
    \item $E(X) [\sqrt{aX^2 +b}]/E$ is ruled.
\end{enumerate}
\end{prop}
\begin{proof}
See \cite[Remark 1.3.5]{GS06} for the equivalence of $(i)$ and $(ii)$.
Since $E$ is rela\-tively algebraically closed in $E(X)[\sqrt{aX^2 +b}]$, $(ii)$ and $(iii)$ are equivalent.
\end{proof}

We recall the following extension of \Cref{RRT} to the case of function fields of conics obtained in \cite{BG21}.
It will be useful in proving our main result.

\begin{thm}\label{BGRRT}
Assume that $v(2) =0$.
Let $F=E(X)[\sqrt{aX^2+b}]$ with $a,b\in \mg{E}$.
Let $w$ be a residually transcendental extension of $v$ to $F$.
Then $Fw/Ev$ is non-ruled if and only if the following conditions hold:
\begin{enumerate}[$(i)$]
\item $v(a),v(b) \in 2vE$, and for $a' \in a\sq E\cap \mg{\mc O}_v$ and $b' \in  b\sq E\cap \mg{\mc O}_v$, the conic $\mc C_{\ovl{a'}, \ovl{b'}}$ over $Ev$ has no rational point.

\item $w|_{E(X)}$ is the Gauss extension of $v$ to $E(X)$ with respect to $cX$ for $c\in \mg E$  with $v(c^2ba^{-1}) = 0$.  

\end{enumerate}
Moreover, if these conditions hold, then $w$ is the unique extension of $w|_{E(X)}$ to $F$,  the extension $w$ of $v$ is unramified and $Fw/Ev$ is the function field of the conic $\mc C_{\ovl{a'}, \ovl{b'}}$ over $Ev$.
\end{thm}
\begin{proof}
Suppose that $Fw/Ev$ is non-ruled.
Then by \cite[Theorem 3.5]{BG21}, we have that $w|_{E(X^2)}$ is the Gauss extension of $v$ to $E(X^2)$ with respect to $b^{-1}aX^2$, whereby $\ovl{b^{-1}aX^2}$ is transcendental over $Ev$.
By \cite[Proposition 3.4]{BG21}, this implies $v(a)$, $v(b)\in 2vE$. 
Let $c\in \mg E$ be such that $v(c^2ba^{-1}) =0$. 
Since $ (\ovl{cX})^2 = (\ovl{c^2ba^{-1}})(\ovl{b^{-1}aX^2}) $ is transcendental over $Ev$, we get that $\ovl{cX}$ is  transcendental over $Ev$. This shows $(ii)$. 
By \cite[Proposition 3.4]{BG21} and \Cref{conicsplit}, we get that for $a' \in a\sq E\cap \mg{\mc O}_v$ and $b' \in  b\sq E\cap \mg{\mc O}_v$, the conic $\mc C_{\ovl{a'}, \ovl{b'}}$ over $Ev$ has no rational point.

For the converse assume that $(i)$ and $(ii)$ hold. From $(ii)$, it follows that $w|_{E(X^2)}$ is the Gauss extension of $v$ to $E(X^2)$ with respect to $b^{-1}aX^2$. 
From $(i)$, using \cite[Proposition 3.4]{BG21} and \Cref{conicsplit}, we get that $Fw/Ev$ is non-ruled.

The uniqueness of $w$ follows by \cite[Corollary 3.6]{BG21}.
\end{proof}

%%%%%%%%%%%%%%%%%%%%%%%%%%%%%%%%%%%%%%%%%%%%%%%%%%%%%%%%%
\section{Ruled valuations on elliptic function fields}%%%
%%%%%%%%%%%%%%%%%%%%%%%%%%%%%%%%%%%%%%%%%%%%%%%%%%%%%%%%%

Let $E$ be a field of characteristic different from $2$ and $3$. Let $a,b \in E$. To these elements, we associate the plane affine curve  $$\mc{E}_{a,b}:\,Y^2\,=\,X^3+aX+b$$
and the quantity $$\Delta_{a,b}\,=\,4a^3+27b^2\,\in \, E\,,$$ which is the discriminant of the cubic polynomial $X^3+aX+b$ and also referred to as the \emph{discriminant of $\mc{E}_{a,b}$}.
Note that $X^3+aX+b$ is separable if and only if $\Delta_{a,b}\neq 0$, and in this case $\mc{E}_{a,b}$ is a smooth elliptic curve.

For the rest of the section,  we fix $a,b \in E$ with $\Delta_{a,b}\neq 0$. 
Let $F$ be the function field of $\mc{E}_{a,b}$ over $E$, that is,
 $$F \,=\, E(X)[\sqrt{X^3+aX+b}]\,.$$
We further fix the element  $$Z\,= \, (aX+b)^{-1}X^3
 \,\in\, E(X)\,.$$

Let $v$ be a valuation on $E$ and let $w$ be a residually transcendental extension of $v$ to $F$. 
By \Cref{equal_values},  $Fw/Ev$ is ruled for any residually transcendental extension $w$ of $v$ to $F$ with $w(Z)\neq 0$. In the following, we consider residually transcendental  extensions $w$ of $v$ to $F$ such that $w(Z) =0$. We will obtain sufficient conditions on $a,b$ and $\Delta_{a,b}$ in this case for $Fw/Ev$ to be ruled.

\begin{lem}\label{a3=b2} 
Let $w$ be a residually transcendental extension of $v$ to $F$ such that $w(Z)=0$. 
Then
$v(a^3) = v(b^2)$ if and only if $ w(X^3) = w(aX) = w(b)$. 
In particular, if $a,b\in \mg{\mc O}_v$, then $w(X) =0$.
\end{lem}

\begin{proof}
Since $w(Z)=0$, we have $w(X^3)=w(aX+b)\geq \min\{w(aX),w(b)\}$.

If $ w(X^3) = w(aX) = w(b)$, then $v(b^2) =w(X^6) =3w(X^2) = v(a^3)$.

If $w(aX) <w(b)$, then $w(X^3) =w(aX+b)=w(aX) <w(b)$ and hence $v(a^3)=w(X^6) < v(b^2)$.

If  $w(aX) >w(b)$,  then $w(X^3) =w(aX+b)=w(b) < w(aX) $ and hence $v(a)>w(X^2)$ and $v(a^3)>w(X^6)=v(b^2)$.

If $w(X^3)>w(aX) = w(b)$, then $w((a^{-1}b)^3) =w(X^3) > w(b)$ and hence $v(a^3) < v(b^2)$.

If $v(a)=v(b)=0$, then $v(a^3)=0=v(b^2)$ and hence $w(X^3)=w(b)=0$.
\end{proof}

\begin{lem}
\label{lemma_with_c}
Assume that $v(2)=0$. Let $w$ be a residually transcendental extension of $v$ to $F$ such that $w(Z)=0$. 
Assume that $wF =wE(X)$ and $\ovl{Z}$ is transcendental over $Ev$.
Then $w(X^3) \in 2vE$.
Let $c\in E$ be such that $w(cX^3) =0$.
Then $Fw/Ev$ is ruled if one of the following holds$\colon$
\begin{enumerate}[$(i)$]
\item $\ovl{cX^3}$ is algebraic over $Ev$.
\item $\ovl{c(aX+b)}$ is algebraic over $Ev$ and   $w(X)\notin 2vE$.
\end{enumerate}
 \end{lem}

 \begin{proof}
 We first claim that if $w(X)\in vE$ then $wE(X) = vE$.
 Since $w(Z) =0 $, we have $w(X^3) = w(aX+b)$.
Let $d\in \mg{E}$ be such that  $w(dX)=0$. 
Then $w(d^{3}(aX+b))=w(d^3X^3)=0$.
We get that  $(\ovl{dX})^3 = \ovl{(d^3(aX+b))} \: \ovl{Z} $, and since $\ovl{Z}$ is transcendental over $Ev$, we get that at least one of the residues $\ovl{dX}$ and $\ovl{d^3(aX+b)}$ is transcendental over $Ev$. 
Since $E(X) =E(dX) = E(d^3(aX+b))$, we conclude by \Cref{gaussextdef}
that $wE(X) =vE$.

Since $\ovl{Z}$ is transcendental, we have $w(Z+1) =0$, and hence 
$$w(Y^2) =w( X^3 + aX+b) = w((aX+b)(Z+1)) = w(aX+b) =w(X^3).$$
Thus $w(X^3) \in 2wF = 2wE(X)$. Furthermore $wE(Z) = vE$.

If $wE(X) = vE$ , then $w(X^3) \in 2wE(X) =2vE$.
Thus we now assume  that $[wE(X): vE]>1$.
In particular, by the above claim,  $w(X)\notin vE$. 
Thus  $w(X) \notin 2vE$, whereby $\frac{1}{2}w(X) \notin vE$.
This shows that $[wE(X): vE]\geq 3$. 
On the other hand, since  $[E(X):E(Z)] = 3$,  using \Cref{funda_ineq}, we get that $[wE(X): vE] \leq 3$. This together shows that $[wE(X): vE] = 3$.
Then $6wE(X) \subseteq 2vE$. 
Since $w(X)\in 2wE(X)$, we conclude that $3w(X) \in 2vE$.

Let $c\in \mg E$ such that $w(c(aX+b)) = w(cX^3) = 0$. Since $3w(X) \in 2vE$, we can choose $d\in \mg E$ such that $w(d^2)=w(c)$. Set $U =d^2(aX+b)$ and  $V = d^2X^3$.
Then $w(U) = w(V) =0$ and $V = UZ$.
Furthermore 
$$d^2Y^2 = d^2X^3 + d^2(aX+b) =V+U.$$

$(i)$ Assume that $\ovl{cX^3}$ is algebraic over $Ev$, then so is $\ovl{V}$.
Since $\ovl {Z}$ is transcendental over $Ev$, we get that $\ovl{U} = \ovl{Z^{-1}}\,\ovl{V}$ is also  transcendental over $Ev$.
Since $E(X) =E(U)$, by \Cref{gaussextdef}, we have $E(X)w =Ev(\ovl{U})$. 
Using \Cref{BG2.2}, we get that 
$Fw=E(X)w(\ovl{dY}) = Ev(\ovl{U})(\sqrt{\ovl{U}+ \ovl{V}})$, whereby $Fw/Ev$ is ruled.

$(ii)$ Assume that $c(aX+b)$ is algebraic over $Ev$, then so is $\ovl{U}$. Also, assume that  $w(X)\notin 2vE$. 
Since $\ovl {Z}$ is transcendental over $Ev$, we get that $\ovl{V} = \ovl{U}\;\ovl{Z}$ is  transcendental over $Ev$. 
It follows by \Cref{gaussextdef} that $wE(V)=vE$ and $E(V)w=Ev(\ovl{V})$. 
Since $w(X)\in 2wE(X)\setminus 2vE$, arguing as before we get that   $3\leq [wE(X):vE]=[wE(X):wE(V)]$. 
Thus $E(X)w=Ev(\ovl{V})$ by \Cref{funda_ineq}. 
Therefore $Ev$ is algebraically closed in $E(X)w$ and hence $\ovl{U} \in Ev$.
Using \Cref{BG2.2}, we get that 
$Fw =E(X)w[\ovl{dY}] = Ev(\ovl{V})[\sqrt{\ovl{V} + \ovl{U}}]$,
whereby $Fw/Ev$ is ruled.
\end{proof}

\begin{lem}\label{Zalgebraic} 
 Assume that $v(6)=0$. Let $w$ be a residually transcendental extension of $v$ to $F$ such that $w(Z) =0$. Set $\ovl{Z} =Z+\mf m_w$ in $Fw$. Then $Fw/Ev$ is ruled in each of the following cases:
\begin{enumerate}[$(i)$]
\item $\ovl{Z}$ is algebraic over $Ev$ and $\ovl{Z} \neq -1$.
\item $v(\Delta_{a,b}) >  \min \{v(a^3), v(b^2)\}$ and $\ovl{Z}$  is transcendental over $Ev$.
\item $v(\Delta_{a,b}) =  \min \{v(a^3), v(b^2)\}$ and $\ovl{Z} =-1$.
\item  $v(\Delta_{a,b}) =\min \{v(a^3), v(b^2)\}$ and $w(X) \notin 2vE$.
\end{enumerate}
\end{lem}

\begin{proof}
Set $Y = \sqrt{X^3+aX+b}$.
 By \Cref{RRT_corollary}, if $F/E$ is ruled, then so is $Fw/Ev$. Hence we may assume that 
$F/E$ is not ruled. 
In particular we have $E(X)\neq F\neq E(Y)$.
Hence $[F:E(X)]=2$. 
If $wE(X)\neq wF$ then $E(X)w=Fw$ by \Cref{funda_quad}, and hence $Fw/Ev$ is ruled. Thus in order to show that $Fw/Ev$ is ruled, we may further assume the following:
\begin{equation} 
wE(X) = wF \mbox{ and } w(Y^2) \in 2wE(X).  \label{equalvaluegroup}
\end{equation}

$(i)$ Suppose that $\ovl{Z}$ is algebraic over $Ev$ and $\ovl{Z} \neq -1$.
 Then $w (Z+1) =0$ and  $\ovl{Z+1}$ is algebraic over $Ev$. 
Note that  $Y^2= (aX+b)\left( Z +1 \right)$.
It follows by \Cref{u_algebraic_implies_ruled} that $Fw/Ev$ is ruled.
 
$(ii)$ Suppose that $v(\Delta_{a,b}) >  \min \{v(a^3), v(b^2)\}$ and $\ovl{Z}$  is transcendental over $Ev$.
Then $v(4a^3+27b^2)=v(\Delta_{a,b})  >  \min \{v(a^3), v(b^2)\}$, whereby $v(a^3) =v(b^2)$. 
By \Cref{a3=b2}, it follows that $w(X^3) =w(aX) =w(b)$.

Furthermore, we have $\ovl{a^3b^{-2}} = -\frac{27}{4}\neq 0$ in $Ev$ and $w(aX+b)=w(X^3)=w(b)$. 
Set $\theta = b^{-1}aX$.
Then $w(\theta +1)=0, w(\theta)\geq 0$ and $\ovl{\theta}\neq -1$.  
We observe that $Z =a^{-3}b^2  (\theta+1)^{-1}\theta^3$.   
Since $\ovl{Z}$ is transcendental over $Ev$, $\ovl{\theta}$ is also transcendental over $Ev$.
We have $E(X) =E(\theta)$ and  $w|_{E(X)}$ is the Gauss extension of $v$ to $E(X)$ with respect to $\theta$.
By \Cref{gaussextdef}, we get that  $ wE(X) = vE$ and $E(X)w = Ev(\ovl{\theta})$. 
Since
 $$Y^2 =a^{-3}b^3( \theta^3 +a^3 b^{-2}(\theta + 1)),$$ we obtain that $w(Y^2) = 3w(a^{-1}b)$ and by \eqref{equalvaluegroup} we get that $wF=wE(X)=vE$ and $w(a^{-1}b) \in 2vE$.
Hence  $a^{-1}b \sq E \cap \mg{\mc O}_v\neq \emptyset$.
Note that $\theta^3+a^3b^{-2}(\theta+1)\in\mc O_w$ and $\ovl{a^3b^{-2}}=-\frac{27}4$,
whereby $\ovl{\theta} ^3 + \ovl{a^3b^{-2}}(\ovl\theta +1) = ( \ovl{\theta} -3) (\ovl{\theta}+ \frac{3}{2})^2$.
We fix $u \in a^{-1}b\sq E \cap \mg{\mc O}_v$ and by \Cref{BG2.2} we obtain that
  $$Fw = Ev \left(\sqrt{\ovl{u} (\ovl\theta-3)}\right).$$
Hence $Fw/Ev$ is rational, and in particular, ruled.

$(iii)$ Suppose that $v(\Delta_{a,b}) =  \min \{v(a^3), v(b^2)\}$ and $\ovl{Z}=-1$.
Hence we have $w(Z+1)>0$. Since $Y^2=(aX+b)(Z+1)$, we obtain that $$w(Y^2)>w(aX+b)=w(X^3).$$
For any $\theta\in E(X)$ we denote by $\Pol_\theta$ the minimal polynomial of $\theta$ over the field $E(Y^2)=E(X^3+aX+b)$ expressed in the variable $T$.
In particular, we have 
$$\Pol_X = T^3+aT+(b-Y^2)\,.$$

We will distinguish three cases according to the comparison between the values $w(aX)$ and $w(b)$.
In view of \Cref{ruled_applicable}, it will be sufficient to find an element $\theta\in E(X)\cap\mg{\mc{O}}_w$  such that $E(X)=E(Y^2)[\theta]$, 
for which $\Pol_{\theta}\in\mc{O}_w[T]$, $\ovl{\theta}$ is a simple root of $\ovl{\Pol_{\theta}}$ and $\ovl{\theta}$ is algebraic over $Ev$.

Assume first that $w(aX)<w(b)$.  
Then $ w(aX)=w(aX+b)=w(X^3)$, whereby  $a^{-1}X^2\in\mg{\mc O}_w$, and  
  $$\overline{a^{-1}X^2}= \overline{(aX)^{-1}X^3}=\overline{(aX+b)^{-1}X^3} =\ovl{Z} = -1  \mbox{ in } Fw. $$
To show that $Fw/Ev$ is ruled, we may use \Cref{scalarext-noresidueext_for_p} and 
assume that there exists $c\in E$ with $c^2=-a^{-1}$. We set $\theta=cX$.
It follows that
$$\Pol_\theta =T^3-T+c^3(b-Y^2)\,.$$
Since we have $w(Y^2)>w(X^3)$ and $w(b)>w(aX)=w(X^3)$, we obtain that $w(Y^2-b)>w(X^3)=-w(c^3)$.
Hence $\Pol_\theta\in\mc{O}_w[T]$ and  $ \ovl{\Pol_\theta} = T^3-T$ in $E(Y^2)w[T]$. 
Thus $\ovl{\Pol_\theta}$ is separable over $Ev$ and $\ovl{\theta}$ is a simple root of $\ovl{\Pol_\theta}$. 
It follows by \Cref{ruled_applicable} that $Fw/Ev$ is ruled.

Assume now that $w(aX)>w(b)$. 
Then $w(b) =w(aX+b)=w(X^3)$. 
In order to show that $Fw/Ev$ is ruled, we may in view of \Cref{scalarext-noresidueext_for_p} assume that $b^{-1}=c^3$ for some $c\in E$.
We set $\theta =cX$. It follows that
 $$\Pol_\theta = T^3+ac^2T+(1-c^3Y^2)\in\mc{O}_w[T]\,.$$

Since $w(Y^2) > w(X^3) =-w(c^3)$ we have $w(c^3Y^2) >0$, 
and since further $w(ac^{-1}) = w(aX) > w(b) =w(X^3) = -w(c^3)$
we have  $w(ac^2)> 0$.
Hence $\Pol_\theta\in\mc{O}_w[T]$ and $\ovl{\Pol_\theta}=T^3+1$ in $E(Y^2)w[T]$.
 Thus $\ovl{\Pol_\theta}$ is separable over $Ev$ and $\ovl{\theta}$ is a simple root of $\ovl{\Pol_\theta}$.
It follows by \Cref{ruled_applicable} that $Fw/Ev$ is ruled.

Assume finally that $w(aX)  =w(b)$. Since $aX+b \neq  0$, we have $a,b \in \mg{E}$. 
Set $\theta= (aX)^{-1}{b}\in\mg{\mc{O}}_w$.
It follows that 
$$\Pol_\theta = T^3+(1-b^{-1}Y^2)^{-1}(T^2+a^{-3}b^2).$$
Since
$w(Y^2)>w(aX+b)\geq w(b)$,
we have $w(b^{-1}Y^2) >0$. 
Since further $w(a^{-3}b^3) =w(X^3) = w(aX+b) \geq w(b)$ we have $w(a^{-3}b^2) \geq 0$.
Hence $\Pol_\theta \in \mc O_w[T]$ and
$\ovl{\Pol_\theta} = T^3+T^2+\ovl{a^{-3}b^{2}}$ in $E(Y^2)w[T]$.

If $\ovl{a^{-3}b^2} =0$, then $\ovl{\Pol_\theta} = T^3+T^2$ and,    
    since $\ovl{\theta} \neq 0$ we have that $\ovl{\theta}$ is a simple root of  $\ovl{\Pol_\theta}$ and it follows by \Cref{ruled_applicable}  that $Fw/Ev$ is ruled.
Suppose that  $\ovl{a^{-3}b^2}\neq 0$. 
Since $v(\Delta_{a,b}) = \min \{v(a^3), v(b^2)\}$, we have that $\ovl{a^{-3}b^2} \neq - \frac{4}{27}$. 
Thus  $\ovl{\Pol_\theta} = T^3+ T^2+\ovl{(a^{-3}b^2)} $  is separable over $Ev$, whereby $\ovl{\theta}$ is a simple root of $\ovl{\Pol_\theta}$. 
 It follows by \Cref{ruled_applicable} that  $Fw/Ev$ is ruled.

$(iv)$ Assume that $v(\Delta_{a,b}) =\min \{v(a^3), v(b^2)\}$ and $w(X) \notin 2vE$. 
In order to show that $Fw/Ev$ is ruled, we may in view of $(i)$ and $(iii)$ assume that  $\ovl{Z}$ is transcendental over $Ev$. By \Cref{lemma_with_c}, we have $3w(X)=w(X^3)\in 2vE$.
Let $c\in \mg E$ be such that $w(X^3) = w(aX+b) = - v(c)$.
If $\ovl{c(aX+b)}$ is transcendental over $Ev$, then by \Cref{gaussextdef}, $wE(X)=vE$ and hence $2w(X)\in 2vE$.
Since $3w(X), 2w(X)\in vE$, we get that $w(X)\in 2vE$, which is a contradiction. Thus $\ovl{c(aX+b)}$ is algebraic over $Ev$. 
It follows from \Cref{lemma_with_c}$(ii)$ that $Fw/Ev$ is ruled.
\end{proof}

\begin{thm}\label{potentialgoodreduction}
Assume that $v(6)=0$, $v(\Delta_{a,b}) = \min \{v(a^3), v(b^2)\}$ and $v(\Delta_{a,b}) \notin 12vE$.
Let $w$ be a residually transcendental extension of $v$ to $F$.
Then $Fw/ Ev$ is ruled.
\end{thm}
\begin{proof}
Set $Z= (aX+b)^{-1}X^3$. 
If $w(Z) \neq 0$, then  $Fw/Ev$ is ruled, by \Cref{equal_values}.
Thus we may assume that $w(Z) =0$.
Using \Cref{Zalgebraic}$(i),(iii)$ and $(iv)$, we may assume that  $\ovl{Z}$ is transcendental over $Ev$ and $w(X) \in 2vE$.
We fix $c\in \mg E$ with $v(c^2X) =0$.
If $\ovl{c^2X}$ is transcendental over $Ev$, then by \Cref{gaussextdef} we have 
$$ 0 = w(c^6(aX+b)) = w(ac^{4} (c^2X) +bc^6) = \min\{w(ac^{4}), w(bc^6)\},$$  
whereby $ v(\Delta_{a,b}) \in 12vE$.
Since $ v(\Delta_{a,b}) \notin 12vE$, we get that $\ovl{c^2X}$ is algebraic over $Ev$.
Hence $\ovl{c^6X^3}$ is also algebraic over $Ev$.
               It follows by \Cref{lemma_with_c}$(i)$ that $Fw/Ev$ is ruled. 
\end{proof}

%%%%%%%%%%%%%%%%%%%%%%%%%%%%%%%%%%%%%%%%%%%%%%%%%%%%%%%%%%%%%%
\section{Non-ruled valuations on elliptic function fields} %%%
%%%%%%%%%%%%%%%%%%%%%%%%%%%%%%%%%%%%%%%%%%%%%%%%%%%%%%%%%%%%%%

Let $E$ be a field and $v$ be a valuation on $E$ such that $v(6) =0$.
As before, we fix two elements $a,b \in E$ with $\Delta_{a,b}\neq 0$. 
Let $F/E$ be the elliptic function field given by $\mc{E}_{a,b}$, that is,  $$F =E(X)[\sqrt{X^3+aX+b}]\,.$$ 
In \Cref{potentialgoodreduction}, we have seen that, if $v(\Delta_{a,b}) = \min \{v(a^3), v(b^2)\}\notin 12vE$, then $Fw/Ev$ is ruled for any residually transcendental extension $w$ of $v$ to $F$.
We now investigate the situation where $v(\Delta_{a,b}) = \min \{v(a^3), v(b^2)\}$ and $v$ has an extension $w$ to $F$ such that $Fw/Ev$ is transcendental and non-ruled.

\begin{thm}
\label{delta_unit} 
Assume that $a,b \in {\mc O}_v$  and $\Delta_{a,b} \in  \mg{\mc O}_v$. 
Let $w$ be a residually transcendental extension of $v$ to $F$. Then $Fw/Ev$ is  non-ruled if and only if $w|_{E(X)}$ is the Gauss extension of $v$ to $E(X)$ with respect to $X$.   
Moreover, in this case, $w$ is the unique extension of $w|_{E(X)}$ to $F$, the extension $w$ of $v$ is unramified and $Fw/Ev$ is the function field of the elliptic curve $\mc{E}_{\ovl{a},\ovl{b}}$ over $Ev$.
\end{thm}

\begin{proof}
Suppose first that $w|_{E(X)}$ is the Gauss extension of $v$ to $E(X)$ with respect to $X$.
By \Cref{gaussextdef}, we have that $wE(X) =vE$, $E(X)w  =Ev(\ovl{X})$ and $w(Y^2) = w(X^3+aX+b) = \min \{v(1), v(a), v(b)\} = 0$, where we are using that $a,b\in\mc{O}_v$ and $4a^3+27b^2=\Delta_{a,b}\in\mg{\mc{O}}_v$.
By \Cref{BG2.2}, we obtain that
$$Fw = E(X)w \left[\,\ovl{Y}\,\right] = Ev(\ovl{X})\left[\sqrt{ \ovl{X}^3+\ovl{a} \ovl{X} + \ovl{b}}\,\right].$$
Since $\Delta_{\ovl{a},\ovl{b}} = \ovl{\Delta_{a,b}} \neq 0$ in $Ev$, we obtain that $Fw/Ev$ is the function field of the elliptic curve $\mc{E}_{\ovl{a},\ovl{b}}$ over $Fv$, so in particular it is non-ruled.

To prove the converse implication, we suppose now that $Fw/Ev$ is non-ruled. 
Set $Z =(aX+b)^{-1}X^3$.
Since $Fw/Ev$ is non-ruled, we have $w(Z) =0$, by \Cref{equal_values}.
Since $v(\Delta_{a,b}) =0 =\min\{v(a^3), v(b^2)\}$, it follows by \Cref{Zalgebraic}  $(i)$ and $(iii)$ that $\ovl{Z}$ is transcendental over $Ev$.
In particular $w(Z)=w(Z+1)=0$.
Since $Y^2=(aX+b)(Z+1)$, we get that $$w(Y^2)=w(X^3) =  w(aX+b) \geq \min \{w(aX), w(b)\}.$$
We claim that $w(X) =0$.

If $w(aX)<w(b)$, then $w(X^3)=w(aX+b)=w(aX)<w(b)$ and hence $w(b^2)>w(X^6)=w(a^3)=w(\Delta_{a,b})=0$, whereby 
 $w(a)=0=w(X)$.
 
If $w(aX)>w(b)$, then $w(X^3)=w(aX+b)=w(b)<w(aX)$ and hence $w(a^3)>w(X^6)=w(b^2)=w(\Delta_{a,b})=0$, whereby  $w(b)=0=w(X)$.

Assume now that $w(aX) = w(b)$. 
Then $w(X) = w(a^{-1}b)$.
Set $V = b^{-1}aX$.
Since $Fw/Ev$ is non-ruled, $wF=wE(X)$.
Otherwise, if $wF\neq wE(X)$ then by \Cref{funda_quad}, $Fw=E(X)w$ and hence $Fw/Ev$ is ruled by \Cref{RRT}. Now \Cref{lemma_with_c}$(i)$ yields that $\ovl{V}^3$ is transcendental over $Ev$. 
Thus $\ovl{V}$ is transcendental over $Ev$, and in particular $w(V+1)=w(V)=0$.
Since $aX+b=b(V+1)$, it follows that $w(b^{3}a^{-3})=w(X^3) = w(aX+b)= w(b)$, whereby $w(b^2)=w(a^3)$.
Since $\min\{v(a^3), v(b^2)\}=v(\Delta_{a,b}) =0$, we conclude that $w(a) = w(b)=0$ and hence $w(X)=0$.

Hence we have in every case that $w(X) =0$.
Since $\ovl{Z} = (\ovl{a} \ovl{X} + \ovl{b})^{-1} \ovl{X}^3$ is transcendental over $Ev$, we get that $\ovl{X}$ is also transcendental over $Ev$. Thus $w|_{E(X)}$ is the Gauss extension with respect to $X$.

Since $[Fw: E(X)w] =2=[F:E(X)]$, it follows by \Cref{funda_quad} that $w$ is the unique extension of $w|_{E(X)}$ to $F$ and $wF=wE(X)=vE$. 
\end{proof}

\begin{cor}
\label{unique_for_good_reduction}
Suppose that $a,b\in \mc{O}_v$ and $\Delta_{a,b}\in\mg{\mc O}_v$.
Then there exists a unique extension $w$ of $v$ to $F$ such that 
$Fw/Ev$ is transcendental and non-ruled. 
Moreover, for this extension, we have $wF=vE$, and 
$Fw/Ev$ is the function field of an elliptic curve defined over $Ev$.
\end{cor}

\begin{proof}
This is clear by \Cref{delta_unit}.
\end{proof}

Using \Cref{delta_unit} we can easily obtain examples of elliptic function fields $F/E$ for which $v(\Delta_{a,b}) =  \min\{ v(a^3), v(b^2)\}$ and an extension $w$ of $v$ to $F$ such that $Fw/Ev$ is transcendental and non-ruled. 
We now give such an example where $v(\Delta_{a,b} ) > \min\{v(a^3), v(b^2\})$. 
The following example stems from \cite[(3.10)]{TVGY06}.

\begin{ex}
Let $E=\mathbb{R}(\!(t)\!)$ and $v$ be the $t$-adic valuation on $E$, which has $vE=\zz$ and $Ev=\mathbb{R}$.
Consider the polynomial $f =(tS-1)(S^2+1)\in E[S]$ and the elliptic curve $\mc{E}: Y^2=f$.
Let $F$ be the function field of $\mc{E}$, that is  $F=E(S)[\sqrt{f}].$
 Let $w$ be a valuation on $F$ such that  $w|_{E(S)}$ is the Gauss extension of $v$ to $E(S)$ with respect to $S$.
Then $\ovl{S} =S+\mf m_w\in Fw$ is transcendental over $Ev$, and 
by \Cref{gaussextdef}, we have $E(S)w = \rr(\ovl{S})$ and $wE(S) =vE$. 
Note that $w (f (S)) =0$. 
Using \Cref{BG2.2}, we get that $wF =vE$ and
\begin{center}
$Fw=E(S)w\left[\sqrt{\ovl{f}}\,\,\right]=\mathbb{R}\left(\ovl{S}\right)\left[\sqrt{-(\ovl{S}^2+1)}\,\right]. $
\end{center}
Hence $Fw/Ev$ is the function field of the the conic $\mc C_{-1,-1}$.
Since this conic has no rational point over $\rr$, using \Cref{conicsplit}, we get that $Fw/Ev$ is not a rational function field. Thus $Fw/Ev$ is non-ruled. 

Substituting $S=tX+\frac{1}{3t}$ in $Y^2 =f$, we get the equation $$Y^2= X^3+\frac{1}{3t^4} \left(3t^2 -1\right)X -\frac{2}{27t^6}\left(9t^2+1\right).$$
Thus $F/E$ is the function field of the elliptic curve $\mc{E}_{a,b}$ where $a=\frac{1}{3t^4} \left(3t^2 -1\right)$ and $b=-\frac{2}{27t^6}\left(9t^2+1\right)$.
We observe that  $\Delta_{a,b} = \frac{4}{t^{10}} (t^2+1)^2$ and further that  $v(\Delta_{a,b}) = -10 >-12=v(a^3) =v(b^2)$.
\end{ex}

\begin{lem}\label{badreductionlemma}
Assume that $v(\Delta_{a,b})>\min\{v(a^3), v(b^2)\}$ and $v(ab) \in 2vE$. 
Then there exists $d\in\mg{E}$ such that $v(d^{12}a^3) = v(d^{12}b^2) = 0$, and letting $\alpha =d^4a$ and $\beta =  d^6b$, we have that $v(\Delta_{\alpha,\beta})>0$ and $F/E$ is the function field of the elliptic curve $\mc{E}_{\alpha, \beta}$ over $E$.
\end{lem}

\begin{proof}
Since $v(\Delta_{a,b}) = v(4a^3+27b^2) >\min\{v(a^3), v(b^2)\}$ we have $ v(a^3) =v(b^2)$.
Thus $v(a) \in 2vE$, and since $v(ab) \in 2vE$, we also have $v(b)\in 2vE$.
It follows that $v(a^3)=v(b^2) \in 12vE$.
We fix $d\in \mg E$ with $v(d^{12}a^3) = v(d^{12}b^2) = 0$.
Then $(d^3Y)^2=(d^2X)^3+ad^4(d^2X)+d^6b$. 
Letting $\alpha =ad^4$ and $\beta =  d^6b$,
we obtain that the desired conditions are satisfied.
\end{proof}

In view of \Cref{unique_for_good_reduction} and \Cref{badreductionlemma}, we turn our attention to residually transcendental extensions of $v$ to $F$ in the case where $a,b \in \mg{\mc O}_v$ and $v(\Delta_{a,b})>0$. 
The following proposition describes a type of residually transcendental extension $v$ on $F/E$ whose residue field extension is possibly non-ruled.

\begin{prop}\label{badreductionexample}
Assume that $a,b \in \mg{{\mc O}}_v$, $v(\Delta_{a,b})>0$ and $v(\Delta_{a,b})\in 2vE$. Let $c\in E$ be such that $v(c^2 \Delta_{a,b}) =0$. Set $S = c(6aX+9b)$ and $u=c^2\Delta_{a,b}$.
Let $w$ be a valuation on $F$ such that $w|_{E(X)}$ is the Gauss extension of $v$ to $E(X)$ with respect to $S$.   
Then $wF =vE$ and $Fw/Ev$ is the function field of the conic $\mc C_{-2\ovl{ab}, \ovl{u}}$. In particular, $Fw/Ev$ is ruled if and only if the conic $\mc C_{-2\ovl{ab}, \ovl{u}}$ over $Ev$ has a rational point.
\end{prop}
\begin{proof}
Since $E(X)=E(S)$ and $w|_{E(X)}$ is the Gauss extension with respect to $S$, we have $E(X)w =Ev(\ovl S)$ and $wE(X) = vE$, by \Cref{gaussextdef}. 
Note that 
$$ (6a)^3 c^2Y^2 = c^{-1}S(S^2 + 9 c^2\Delta_{a,b})- 27b (S^2 +c^2\Delta_{a,b} ) .$$
Since $w(c^2\Delta_{a,b}) =0$, we have $w(S^2 +c^2\Delta_{a,b} )=w (S(S^2 + 9 c^2\Delta_{a,b})) =0$, and since $v(c) < 0=w((6a)^3)$, we obtain that  $w(c^2Y^2) =  w(S^2 +c^2\Delta_{a,b} ) = 0$ and 
$\ovl{(4a^2cY)}^2 = -2\ovl{ab} (\ovl{S}^2 + \ovl{c^2\Delta_{a,b}})\in \mg{Ev(\ovl{S})}\setminus \sq{Ev(\ovl{S})}$. 
By \Cref{BG2.2}, we get that 
$$Fw = 
Ev(\ovl{S}) \left[\sqrt{ -2\ovl{ab}(\ovl{S}^2 + \ovl{c^2\Delta_{a,b}} ) }\right].$$  
Hence $Fw/Ev$ is the function field of the conic $\mc{C}_{-2\ovl{ab}, -2\ovl{abu}}$, which is isomorphic to the function field of the conic ${\mc C}_{-2\ovl{ab}, \ovl{u}}$.
Hence the final statement follows from \Cref{conicsplit}.
\end{proof}

\begin{prop}\label{unique_for_bad_reduction}
Assume that $ v(\Delta_{a,b}) > v(a^3) =v(b^2)$.
Let $w$ be a residually transcendental extension of $v$ to $F$. 
Then $Fw/Ev$ is non-ruled if and only if the following conditions hold: 
\begin{enumerate}[$(i)$]
\item  $v(ab), v(\Delta_{a,b}) \in 2vE$, and for $u_1 \in  ab \sq E \cap \mg {\mc O}_v$ and $u_2 \in \Delta_{a,b} \sq E \cap \mg {\mc O}_v$, the conic $\mc{C}_{-2 \ovl{u_1}, \ovl{u_2}}$ does not have a rational point over $Ev$.
\item  $w|_{E(X)}$  is the Gauss extension with respect to $c(6aX+9b)$ for $c\in \mg E$ such that  $v(c^2 \Delta_{a,b}) = 0$.
\end{enumerate}
Moreover, if these conditions hold, then $w$ is the unique extension of $w|_{E(X)}$ to $F$,  the extension $w$ of $v$ is unramified and $Fw/Ev$ is the function field of the conic $\mc{C}_{-2 \ovl{u_1}, \ovl{u_2}}$ over $Ev$.
\end{prop}

\begin{proof}
Suppose first that $(i)$ and $(ii)$ hold.  
By \Cref{badreductionlemma}, there exists $d\in \mg E$ with $v(d^4a) = v(d^6b) =0$.
Let $\alpha =ad^4$ and $\beta =  d^6b$.
Then $\Delta_{\alpha, \beta}=d^{12}\Delta_{a,b}$, $v(\Delta_{\alpha,\beta})>0$ and $F/E$ is the function field of the elliptic curve $\mc{E}_{\alpha, \beta}$ over $E$.
In view of $(ii)$, we fix $c\in \mg E$ such that $v(c^2 \Delta_{a,b}) =0$ and $w|_{E(X)}$ is the Gauss extension with respect to $c(6aX+9b)$.
Taking $ X' = d^2X$, we get that $E(X) =E(X')$ and $c(6aX+9b) = d^{-6}c (6 \alpha X' + 
 9 \beta)$. 
For $u_1=\alpha\beta$ and $u_2= c^2 d^{-12}\Delta_{\alpha,\beta}=c^2 \Delta_{a,b}$, 
we have $v(u_1)=v(u_2)=0$, and we obtain by \Cref{badreductionexample} that $Fw/Ev$ is the function field of $\mc{C}_{-2\ovl{u_1}, \ovl{u_2}}$ over $Ev$.
We conclude by $(i)$ and \Cref{{conicsplit}}
 that $Fw/Ev$ is non-ruled.

Suppose now conversely that $Fw/Ev$ is non-ruled. 
Set $Z=  (aX+b)^{-1}X^3$.
By \Cref{equal_values}, we have $w(Z) =0$.
Set $\ovl{Z} =Z+\mf m_w$.
By \Cref{Zalgebraic}, $(i)$ and $(ii)$, we obtain that  $\ovl{Z} = -1$, whereby $w(Z+1)>0$.
Then $Y^2 =(aX+b)(Z+1)$, and it follows that $w(Y^2) > w(aX+b) =w(X^3)$.
  
Since $w(a^3) =w(b^2)$, by \Cref{a3=b2}, we further have $w(X^3) = w(aX) =w(b)$.
Set $\theta = b^{-1}aX$.
The minimal polynomial of $\theta$ over the field $E(Y^2)$ expressed in the variable $T$ is given by
$$ \Pol_\theta = T ^3 + b^{-2}a^3(T +1 -b^{-1} Y^2 )\,.$$
Since $w(Y^2) >w(X^3) =w(b)$ and  $v(\Delta_{a,b}) > v(a^3) =v(b^2)$,  we have $\ovl{b^{-1}Y^2}=0$ and $\ovl{b^{-2}a^3} = -\frac{27}{4}$ in $Fw$.
It follows that $\ovl{\theta}$ is a root of 
$$\ovl{\Pol_\theta} = T^3 +\ovl{b^{-2}a^3}(T +1) = T^3-\hbox{$\frac{27}{4}$} (T +1) = (T-3)\left(T+\hbox{$\frac{3}{2}$}\right)^2 .$$
Hence $\ovl{\theta} = 3$ or $\ovl{\theta} = -\frac{3}{2}$ in $Fw$.

Since $Fw/Ev$ is non-ruled, \Cref{ruled_applicable} excludes the possibility that $\ovl{\theta}= 3$.
Therefore $\ovl{\theta} = -\frac{3}{2}$.
We set $S = 6\theta+9 = b^{-1}(6aX+9b)$.
Then $E(X) =E( S)$ and $w(S)> 0$.
In particular $w( S-27) =w( S-3) =0 $.
We have 
\begin{eqnarray*}
(36a^2b^{-1})^2 Y^2  & = & 6ab((S-27) S^2 + 9( S-3) b^{-2}\Delta_{a,b}  )\\
 & = & 6ab( S( S^2 + 9b^{-2}\Delta_{a,b})- 27(S^2 +b^{-2}\Delta_{a,b})).
\end{eqnarray*}
Hence $F\simeq E( S)[\sqrt{f_1+f_2}]$ for $f_1= 6ab( S-27) S^2$ and $f_2=54ab( S-3) b^{-2}\Delta_{a,b}$. 
For $i\in\{1,2\}$, we have that $E( S)(\sqrt{f_i})/E$ is a rational function field. 
Since $F\simeq E( S)(\sqrt{f_1+f_2})$ and $Fw/Ev$ is non-ruled, we conclude by \Cref{RRT_corollary} and \Cref{equal_values_ruled} that $w(f_1)=w(f_2)$.

Hence $w( S^2) = w(b^{-2}\Delta_{a,b})$ and consequently $w( S^2+ b^{-2}\Delta_{a,b})\geq w( S^2)$.
We claim that $w( S^2+ b^{-2}\Delta_{a,b})= w( S^2)$.

Suppose on the contrary that $w( S^2+ b^{-2}\Delta_{a,b})> w( S^2)$.
Then $-\ovl{(b S)^{-2}\Delta_{a,b}} = 1$. 
Let $\delta=\sqrt{-\Delta_{a,b}}$.
In order to obtain a contradiction with the hypothesis that $Fw/Ev$ is non-ruled, we may by \Cref{scalarext-noresidueext_for_p} extend $w$ to $F(\delta)$ and then replace $E$ by $E(\delta)$ and $F$ by $F(\delta)$.
We may thus assume that $\delta\in E$.

We have 
\begin{equation*}
S^3-27S^2+9b^{-2}\Delta_{a,b} S-27b^{-2}\Delta_{a,b} - (6ab^{-1})^3Y^2 = 0
\,. \label{cubicequation*}
\end{equation*}
Since $w(S)= w(b^{-1}\delta) > 0$,  we have $w(b^{-2}\Delta_{a,b} S)=w (S^3)>w(S^2)=w(b^{-2}\Delta_{a,b})$ and 
$w(S^2 +b^{-2}\Delta_{a,b}) > w(S^2) = w (b^{-2}\Delta_{a,b})$.
Hence
$w ((6ab^{-1})^3Y^2) > w (b^{-2}\Delta_{a,b})$, whereby $b^{-1}\Delta_{a,b}^{-1}(6a)^3 Y^2\in\mfm_w$.

Set $\theta_0 = \delta^{-1}b S$.
Then $E(X) =E(\theta_0)$ and  $w(\theta_0) =0$.
The minimal polynomial of $\theta_0$ over $E(Y^2)$ in the variable $T$ is $$\Pol_{\theta_0} = T^3-27\delta^{-1}bT^2-9T+27\delta^{-1}b+(6a)^3\Delta_{a,b}^{-1}\delta^{-1}Y^2.$$ 
We set $Q_{\theta_0} = b^{-1}\delta \Pol_{\theta_0}\in  E(Y^2)[T]$ and obtain that
$$Q_{\theta_0}=   b^{-1}\delta  T^3 - 27 T^2 - 9 b^{-1}\delta  T + 27 + b^{-1}\Delta_{a,b}^{-1}(6a)^3Y^2\,.$$
Since $b^{-1}\delta\in\mfm_w$ and $b^{-1}\Delta_{a,b}^{-1}(6a)^3 Y^2\in\mfm_w$, it follows that $Q_{\theta_0} \in \mc O_w[T]$ and $\ovl{Q_{\theta_0}}= -27(T^2 - 1)$ in $E(Y^2)w[T]$. 
In particular the residue polynomial $\ovl{Q_{\theta_0}}$ is separable and splits into linear factors.
By \Cref{simple_ext}, we conclude that
 $wE(X)=wE(Y^2)$ and  $E(X)w=E(Y^2)w$.
It follows by \Cref{ruled} that $Fw/Ev$ is ruled, which is a contradiction.

Thus we have established that $w(S ^2 +b^{-2}\Delta_{a,b} ) = w( S ^2)$.
Note that $$F\simeq E( S)[\sqrt{ S g_1 +g_2}]$$  for $g_1=6ab ( S^2 + 9b^{-2}\Delta_{a,b})$ and $g_2 = - 2\cdot 9^2 ab ( S^2 +b^{-2}\Delta_{a,b})$.
Since $w(S) >0 $, we obtain that $w( Sg_1) \geq w( S^3)> w( S^2) =w(g_2)$. 
It follows by \Cref{f_and_g} that $Fw = F'w'$ for some extension $w'$ of $w|_{E(X)}$ to $F' = E( S)\left[\sqrt{g_2}\right]$.
We now conclude by \Cref{BGRRT} that $(i)$ and $(ii)$ hold, $w$ is an unramified extension of $v$ and $Fw/Ev$ is the function field of the conic $\mc{C}_{-2 \ovl{u_1}, \ovl{u_2}}$ over $Ev$.
\end{proof}

%%%%%%%%%%%%%%%%%%%%%%%%%%%%%%%%%%%%%%%%%%%%%%%%%%%%%%%%%
\section{Reductions types and non-ruled residue field extensions}%%%%%%%%%%%
%%%%%%%%%%%%%%%%%%%%%%%%%%%%%%%%%%%%%%%%%%%%%%%%%%%%%%%%%

Let $E$ be a field and let $v$ be a valuation on $E$ with $v(6) =0$. 
Let $F/E$ be an elliptic function field.
We say that $F/E$ \emph{is of good reduction with respect to $v$} if $F\simeq E(\mc{E}_{a,b})$ for some  $a,b\in\mc{O}_v$ with $\Delta_{a,b}\in\mg{\mc O}_v$.
In the following proposition, we characterize elliptic function fields with good reduction with respect to $v$.

\begin{prop}\label{reductiontypes}
Let $a,b\in E$ be such that $\Delta_{a,b}\neq 0$ and let $F/E$ be the function field of $\mc{E}_{a,b}$. Then $F/E$ is of good reduction with respect to $v$ if and only if $v(\Delta_{a,b})=\min\{v(a^3),v(b^2)\}\in 12 vE$. 
\end{prop}

\begin{proof}
Assume that $v(\Delta_{a,b})=\min \{v(a^3), v(b^2)\}\in 12vE$.
Let $d\in E^{\times}$ be such that $v(d^{12}\Delta_{a,b})=0$. 
Set $a' =d^4a$ and $b'=d^6b$. Then $E(\mc{E}_{a',b'})\simeq E(\mc{E}_{a,b})\simeq F$,  $a', b'\in {\mc O}_v$ and $\Delta_{a',b'} \in  \mg{\mc O}_v$.
Hence $F/E$ is of good reduction with respect to $v$.

Assume now that $F/E$ is of good reduction with respect to $v$. 
By \Cref{unique_for_good_reduction}, there exists a residually transcendental  extension $w$  of $v$ to $F$ such that $Fw/Ev$ is function field of an elliptic curve over $E$. 
In particular, $Fw/Ev$ is transcendental and neither ruled nor the function field of a conic.
We conclude by \Cref{unique_for_bad_reduction} and \Cref{potentialgoodreduction} that $v(\Delta_{a,b}) = \min\{v(a^3),v(b^2)\} \in 12 vE$.
\end{proof}

We say that $F/E$ \emph{is of potential good reduction with respect to $v$} if $F=E(\mc{E}_{a,b})$ for certain $
a,b\in\mg{E}$ with $v(\Delta_{a,b})=\min \{v(a^3), v(b^2)\}$.

\begin{thm}
 \label{main_result}
Let $a,b\in E$ such that $\Delta_{a,b}\neq 0$ and let $F/E$ be the function field of $\mc{E}_{a,b}$. 
There is at most one extension $w$ of $v$ to $F$ such that $Fw/Ev$ is transcendental and non-ruled.
Moreover, if such an extension $w$ of $v$ to $F$ exists, then it is unramified  and one of the following holds: 
 \begin{enumerate}[$(i)$]
 \item $F/E$ is of good reduction with respect to $v$ and  $Fw/Ev$ is the function field of an elliptic curve.   
  \item $F/E$ not of potential good reduction with respect to $v$, 
  $v(ab), v(\Delta_{a,b})\in 2vE$ and  $Fw/Ev$ is the function field of a conic having no $Ev$-rational point.
 \end{enumerate}
\end{thm}

\begin{proof}  
If $v(\Delta_{a,b})=\min\{v(a^3),v(b^2)\}\notin 12vE$, then it follows 
by \Cref{potentialgoodreduction} that $Fw/Ev$ is ruled for every residually transcendental extension $w$ of $v$ to $F$. 

Assume now that $v(\Delta_{a,b})=\min\{v(a^3),v(b^2)\}\in 12 vE$. 
By \Cref{reductiontypes}, $\mc{E}_{a,b}$ is of good reduction with respect to $v$.
By \Cref{unique_for_good_reduction}, there exists a unique extension $w$ of $v$ to $F$ such that $Fw/Ev$ is transcendental and non-ruled, and for this valuation $w$, the residue field extension  $Fw/Ev$ is the function field of an elliptic curve.

Assume finally that $v(\Delta_{a,b})>\min\{v(a^3),v(b^2)\}$. 
Hence $\mc{E}_{a,b}$ is not of potential good reduction
with respect to $v$.
By \Cref{unique_for_bad_reduction}, there exists at most one extension $w$ of $v$ to $F$ such that $Fw/Ev$ is transcendental and non-ruled, and if such an extension $w$ exists, then $v(ab), v(\Delta_{a,b})\in 2vE$ and  $Fw/Ev$ is the function field of a conic having no rational point over $Ev$. 
\end{proof}

Elliptic function fields $F/E$ are function fields of genus one curves having a rational point over $E$. 
Finally, we give an example of a field $E$ with a valuation $v$ and a function field $F/E$ of genus one such that there exist two residually transcendental extensions of $v$ to $F$ for which the corresponding residue field extensions are non-ruled.

The following example stems from \cite[Example 5.12]{BVG09}. There it is used for a discussion of a problem of sums of squares. For the relation of that topic with the problem on valuation extensions, we refer to the introduction of \cite{BGr}.

\begin{ex}
Let $E=\mathbb{R}(\!(t)\!)$ and $v$ be the $t$-adic on $E$ with value group $\zz$.
Then $Ev=\mathbb{R}$.
Consider the polynomial $f = -(X^2+t^2)(X^2+1) \in E[X]$ and the curve $\mc E:  Y^2 = f$.
Let $F$ be the function field of $\mc E$, that is,  
$F=E(X)\left[\sqrt{f}\right]$.
Then $F/E$ is a regular function field of genus one.

Let $w_0$ denote the Gauss extension of $v$ to $E(X)$ with respect to $X$. 
Then $E(X)w_0=\rr(\ovl{X})$.
Since $w_0(f) =0$ and $\ovl {f} =-\ovl{X}^2(\ovl{X}^2-1) \notin \sq{\rr(\ovl{X})}$,  it follows by \Cref{BG2.2} that  $w_0$ extends uniquely to a valuation $w$ on $F$ such that $wF=wE(X)=vE=\zz$ and
$$Fw=E(X){w}\left[\sqrt{\ovl{f}}\right]=\mathbb{R}\left(\ovl{X}\right)\left[\sqrt{-(\ovl{X}^2+1)}\right].$$

Now let $w_0'$ denote the Gauss extension of $v$ to $E(X)$ with respect to the variable $X'=t^{-1}X$.
Note that $w(X)=0$ and $w'(X)=v(t)=1$. Hence $w$ and $w'$ are not equivalent. 
Then $E(X)=E(X')$ and $F=E(X')[\sqrt{f'}]$ for $f' = t^{-2}f = -(X'^2+1)(t^2X'^2+1)$. Thus we obtain similarly that $w_0'$ extends uniquely to a valuation $w'$ on $F$ such that  $w'F=\zz$ and 
$$Fw'=\mathbb{R}\left(\overline{X'}\right)\left[\sqrt{-(\overline{X'}^2 +1)}\right]\,.$$

Hence both function fields $Fw/Ev$ and $Fw'/Ev$ are isomorphic to the function field of the conic $\mc C_{-1,-1}$  over $Ev =\rr$.
Since $\mc C_{-1,-1}$ has no rational point over $\rr$, we get that $Fw/Ev$ and $Fw'/Ev$ are non-ruled. 
\end{ex}

\subsection*{Acknowledgments} 
The second and third named authors thank  Prof. Anupam K. Singh for supporting  their research visits to IISER Mohali and IISER Pune respectively, during this work. We would further like to thank the referee for helpful comments and corrections.
\newpage

\bibliographystyle{amsalpha}

\end{document}